\documentclass[final,12pt,reqno]{amsart}

\usepackage{amsmath,amssymb,amscd,amsthm}
\usepackage{latexsym}
\usepackage{graphicx}
\usepackage{color}
\usepackage{showkeys}
\usepackage{epstopdf}

\newtheorem{thm}{Theorem}[section]

\newtheorem{corollary}[thm]{Corollary}
\newtheorem{prop}[thm]{Proposition}
\newtheorem{define}[thm]{Definition}

\newtheorem{lemma}[thm]{Lemma}

\newcommand{\al}{\alpha}
\newcommand{\pp}{\partial}
\newcommand{\om}{\omega}
\newcommand{\na}{\nabla}
\newcommand{\De}{\Delta}

\newcommand{\ral}{\mathcal R_\alpha}
\newcommand{\La}{\Lambda}

\newcommand{\rd}{{\mathbf R^d}}
\newcommand{\f}[2]{\frac{#1}{#2}}
\newcommand{\ga}{\gamma}

\numberwithin{equation}{section}
\begin{document}
\title[2D Boussinesq Equations]{A global regularity result for the 2D Boussinesq equations with critical dissipation}

\author[A. Stefanov and J. Wu]{Atanas Stefanov$^{1}$ and  Jiahong Wu$^{2}$}

\address{$^1$ Department of Mathematics,
University of Kansas,
1460 Jayhawk Blvd,
Lawrence, Kansas 66045-7594}

\email{stefanov@ku.edu}

\address{$^2$ Department of Mathematics,
Oklahoma State University,
Stillwater, OK 74078}

\email{jiahong.wu@okstate.edu}

\vskip .2in
\begin{abstract}
This paper examines the global regularity problem on the two-dimensional
incompressible Boussinesq equations with fractional dissipation, given
by $\Lambda^\alpha u$ in the velocity equation and by
$\Lambda^\beta \theta$ in the temperature equation,
where $\Lambda=\sqrt{-\Delta}$ denotes the Zygmund operator.
We establish the global existence and smoothness of classical solutions
when $(\alpha,\beta)$ is in the critical range: $\alpha>\f{\sqrt{1777}-23}{24} =0.798103..$, $\beta>0$
and $\alpha+ \beta =1$. This result improves the previous work of \cite{JMWZ}
which obtained the global regularity for $\alpha> \frac{23-\sqrt{145}}{12} \approx 0.9132$,  $\beta>0$
and $\alpha+ \beta =1$.
\end{abstract}

\maketitle

\vskip .2in
\section{Introduction}

This paper aims at the global regularity issue on
the two-dimensional (2D) Boussinesq equations with fractional dissipation
\begin{eqnarray}\label{BQE}
\begin{cases}
\partial_t u + u\cdot \nabla u + \nu\, \Lambda^\alpha u
=- \nabla p + \theta \mathbf{e}_2, \qquad x\in \mathbb{R}^2, \,\, t>0, \\
\nabla \cdot u=0, \qquad x\in \mathbb{R}^2, \,\, t>0, \\
\partial_t \theta + u\cdot \nabla \theta
+ \kappa\, \Lambda^{\beta}\theta =0,\qquad x\in \mathbb{R}^2, \,\, t>0,\\
u(x,0) =u_0(x),\,\, \theta(x,0) =\theta_0(x), \qquad x\in \mathbb{R}^2,
\end{cases}
\end{eqnarray}
where $u=u(x,t)$ denotes the 2D velocity, $p=p(x,t)$ the pressure,
$\theta=\theta(x,t)$ the temperature, $\mathbf{e}_2$ the unit vector in the vertical
direction, and $\nu>0$, $\kappa>0$, $0<\alpha \le 2$ and $0<\beta\le 2$
are real parameters. Here $\Lambda= \sqrt{-\Delta}$ represents the Zygmund
operator with $\Lambda^\alpha$ being defined through the Fourier transform,
$$
\widehat{\Lambda^\alpha f}(\xi) = |\xi|^\alpha \,\widehat{f}(\xi),
$$
where the Fourier transform is given by
$$
\widehat{f}(\xi) = \int_{\mathbb{R}^2} e^{-i x\cdot \xi} \, f(x) \,dx.
$$
When $\alpha=\beta=2$, (\ref{BQE}) reduces to the standard 2D Boussinesq equations with Laplacian dissipation. The standard Boussinesq equations
model geophysical flows such as atmospheric fronts and oceanic
circulation, and also play an important role in the study of
Raleigh-Bernard convection (see, e.g., \cite{Con_D,Gill,Maj,Pe,Doering1,Doering2}).

\vskip .1in
Although (\ref{BQE}) with fractional dissipation appears to be a purely 
mathematical generalization, (\ref{BQE}) may be physically relevant. 
First, closely related equations such as the surface quasi-geostrophic
equation model important geophysical phenomena 
(see, e.g., \cite{CMT, HPGS,Pe}). Second,
there are geophysical circumstances in which the Boussinesq equations with fractional Laplacian may arise. Flows in the middle atmosphere traveling upward undergo changes due to the changes of atmospheric properties, although the incompressibility and Boussinesq approximations
are applicable. The effect of kinematic and thermal diffusion is attenuated by the thinning of atmosphere. This anomalous attenuation can be modeled by using the space fractional Laplacian (see \cite{Gill, Caputo}). Third, 
it may be possible to derive the Boussinesq equations with fractional dissipation from the Boltzmann-type equations using suitable rescalings.
A recent paper \cite{Hit} derives the fractional Stokes and Stokes-Fourier 
systems as the incompressible limit of the Boltzmann equation. 

\vskip .1in
Mathematically the 2D Boussinesq equations serve as a lower dimensional
model of the 3D hydrodynamics equations. In fact, the 2D Boussinesq
equations retain some key features of the 3D Navier-Stokes and the
Euler equations such as the vortex stretching mechanism. As pointed out
in \cite{MB}, the inviscid Boussinesq equations can be identified with
the 3D Euler equations for axisymmetric flows. It is hoped that the study
on the 2D Boussinesq equations will shed light on the mysterious global
existence and smoothness problem on the 3D Navier-Stokes and Euler equations. The generalization to include the fractional dissipation facilitates this purpose by allowing the simultaneous study of a whole 
parameter family of equations.

\vskip .1in
One main pursuit in the study of (\ref{BQE}) has been to obtain the global
regularity of its solutions for the smallest $\alpha$
and $\beta$. Intuitively, the smaller $\alpha$ and $\beta$ are, the harder
the global regularity problem is. When there is no dissipation, namely
$\nu=\kappa=0$ in (\ref{BQE}), the global regularity problem remains  open. The standard
idea of proving the global {\it a priori} bounds in Sobolev spaces fails.
Potential finite time singularities have been explored from different
perspectives including boundary effects and 1D models (\cite{LuoHou,CKY,CHKLSY,SaWu}).

\vskip .1in
At the other extremum,
when $\nu>0$, $\kappa>0$, $\alpha=\beta=2$, the global regularity can
be easily obtained, in
a similar fashion as for the 2D Navier-Stokes equations
(\cite{DoeringG, MB}). It is natural to examine (\ref{BQE}) with intermediate dissipation, which has attracted considerable attention
in the last few years (\cite{ACW10,ACW11,ACWX,CaoWu1,Ch,CV,Jia,Jiu,Dan,DP3,Hmidi,HKR1,HKR2,HL,
JMWZ,JWYang,KRTW,LaiPan,LLT,MX,Oh,Wu,Wu_Xu,Wu_Xu_Ye,Xu,YEX,Zhao}).
\cite{Ch} and \cite{HL} have shown that one
full Laplacian dissipation in (\ref{BQE}) is sufficient for the global
regularity. More precisely, (\ref{BQE}) with $\alpha=2$ and $\kappa=0$
or with $\beta=2$ and $\nu=0$ always possesses global
classical solutions.

\vskip .1in
More recent work further reduces the values of $\alpha$ and $\beta$ and
existing research appears to indicate that one-derivative dissipation is critical. Here one-derivative
dissipation refers to the case when $\alpha+\beta=1$ in (\ref{BQE}). For the
convenience of description, $\alpha+ \beta=1$ is referred to as the
critical case, $\alpha+ \beta>1$ as the subcritical case while
$\alpha+\beta<1$ as the supercritical case. To position our work in a suitable
context, we describe some recent results for these three cases.

\vskip .1in
We start with the subcritical case.  Even this case is not easy.
The global regularity has so far been
established only for three subcritical cases (\cite{CV,MX,YangJW}).
The global existence and regularity problem for the critical case
is more difficult. Two particular critical cases, $\alpha=1,\,\kappa=0$
and $\beta=1, \nu=0$, were studied by \cite{HKR1} and \cite{HKR2},
which introduced a combined quantity of the vorticity and the Riesz transform of the temperature and were able to establish the global regularity for both cases. For the more general critical case when the one derivative dissipation is split between the velocity equation
and the temperature equation, the situation becomes more complex.
The general critical case was recently dealt with by Jiu, Miao, Wu and Zhang
and a global regularity result was obtained \cite{JMWZ}. By reducing the global regularity issue
on the critical Boussinesq system to a parallel problem for an active
scalar equation with critical dissipation or, more precisely, the critical surface
quasi-geostrophic (SQG) equation and taking advantage of the recent
advances on the SQG equation, Jiu, Miao, Wu and Zhang obtained the global
regularity in the critical regime: $\alpha+ \beta=1$ and
$\alpha >\alpha_0$, where $\alpha_0 =\frac{23-\sqrt{145}}{12} \approx 0.9132$. Attempts have also been made to
go beyond the critical case and the global regularity has been established when
the dissipation is logarithmically more singular than the critical case (\cite{Hmidi,KRTW}).
The global well-posedness problem for the supercritical case $\alpha+ \beta<1$
is completely open. The only result currently available is the eventual
regularity of weak solutions to (\ref{BQE}) with $\alpha+\beta<1$ and
$\alpha >\alpha_0$ \cite{JWYang}.

\vskip .1in
This paper establishes the global existence and regularity of classical
solutions to (\ref{BQE}) when $\alpha$ and $\beta$ are in the critical
range: $\alpha+\beta=1$ and $1>\alpha >\f{\sqrt{1777}-23}{24} =0.798103..$. This result improves the
work of Jiu, Miao, Wu and Zhang
\cite{JMWZ} by allowing $\alpha$ to vary in a bigger interval, by keeping the relation $\al+\beta=1$. The precise statement of our result
is given in the following theorem.

\begin{thm}\label{main}
Consider (\ref{BQE}) with $(u_0, \theta_0)\in H^\sigma(\mathbb{R}^2)$ for $\sigma>2$.
If the parameters in (\ref{BQE}) satisfies
$$
\nu>0, \quad\kappa>0, \quad  0.798103.. =\f{\sqrt{1777}-23}{24} <\alpha<1, \quad \alpha+\beta=1,
$$
then (\ref{BQE}) has a unique global solution $(u, \theta)$ satisfying, for any $T>0$,
$$
(u, \theta) \in C([0, T]; H^\sigma(\mathbb{R}^2)).
$$
\end{thm}

\vskip .1in
We outline the main idea in the proof of this theorem and explain how we
improved \cite{JMWZ}. A large portion of the efforts are devoted to obtaining global
{\it a priori} bounds for $(u, \theta)$. Due to $\nabla \cdot u =0$, the
$L^2$-level global bounds for $(u, \theta)$ follow from easy energy estimates,
$$
\|\theta(t)\|_{L^q} \le \|\theta_0\|_{L^q} \,\,\,\mbox{for $q\in [1,\infty]$},
$$
\begin{equation} \label{L2theta}
\|\theta(t)\|_{L^2}^2 + 2 \kappa \int_0^t \|\Lambda^{\frac{\beta}{2}} \theta(\tau)\|_{L^2}^2\, d\tau = \|\theta_0\|_{L^2}^2,
\end{equation}
$$
\|u(t)\|_{L^2}^2 + 2 \nu \int_0^t \|\Lambda^{\frac{\alpha}{2}}
u(\tau)\|_{L^2}^2\, d\tau
\le (\|u_0\|_{L^2}  + t \|\theta_0\|_{L^2})^2.
$$
Naturally the next target is the global $H^1$-bound for $u$ or equivalently
global $L^2$-bound for the vorticity $\omega=\nabla\times u$, which satisfies
\begin{eqnarray}\label{Veq}
\begin{cases}
\partial_t \om +u\cdot \nabla \om + \nu \Lambda^\alpha \om= \partial_{1}\theta , \\
u=\nabla^\perp \psi, \quad \Delta \psi = \omega \qquad \mbox{or}\quad
u =\nabla^\perp \Delta^{-1} \om.
\end{cases}
\end{eqnarray}
Due to the presence of the ``vortex stretching" term $\pp_1 \theta$, direct
energy estimates do not yield the desired global bound for $0<\alpha<1$. For notational convenience, we set $\nu=\kappa=1$ in (\ref{BQE}) throughout the rest of this paper. The
strategy is to hide $\pp_1 \theta$ by considering the combined quantity
$$
G = \om - \mathcal{R}_\alpha \theta \quad \mbox{with}\quad \mathcal{R}_\alpha = \Lambda^{-\alpha} \pp_1.
$$
It is easy to check that $G$ satisfies
\begin{equation}\label{Geqin}
\partial_t G + u\cdot \nabla G + \Lambda^\alpha G= [\mathcal{R}_\alpha, u\cdot\na]\theta +\Lambda^{\beta-\alpha}\partial_{1}\theta.
\end{equation}
Here we have used the standard commutator notation
$$
[\mathcal{R}_\alpha, u\cdot\na]\theta
= \mathcal{R}_\alpha(u\cdot\na\theta) - u\cdot\na \mathcal{R}_\alpha\theta.
$$
Although (\ref{Geqin}) appears to be more complicated than
the vorticity equation, the commutator term $[\mathcal{R}_\alpha, u\cdot\na]\theta$
is less singular than $\pp_1\theta$ in the vorticity equation. In fact, we are
able to show the global bound
for $\|G\|_{L^2}$ whenever $\alpha>\frac34$ and $\alpha+ \beta=1$.
The major contribution of this paper is on the global $L^6$-bound for $G$.
Previously, for $\alpha>\frac45$ and $\alpha+ \beta=1$,  \cite{JMWZ} obtained a global bound for $\|G\|_{L^q}$ for $q$ in the range
\begin{equation}\label{qrange}
2 \le  q < q_0 \equiv \frac{8-4\alpha}{8-7\alpha}.
\end{equation}
Obviously $q_0\in (2,4)$ when $\alpha\in (\frac45, 1)$. We are able to significantly enlarge the range of $q$. More precisely, we prove the
following proposition.

\begin{prop} \label{prop:Lq}
Let $0<\beta<1$ and $\alpha+\beta=1$, with
$$
1>\al> \f{\sqrt{1777}-23}{24}= 0.798103..
$$
Let $(u_0, \theta_0)$ be as specified in Theorem \ref{main} and let $(u, \theta)$ be the corresponding smooth solution of (\ref{BQE}). Assume $G$
satisfies (\ref{Geqin}). Then, for any $T>0$ and $t\le T$,
$$
\|G(t)\|_{L^q} \le C \quad\mbox{for}\quad 2\le q\le 6,
$$
where $C$ is a constant depending on $T$ and
the initial data $(u_0, \theta_0)$.
\end{prop}

The proof of Proposition \ref{prop:Lq} involves the decomposition of
the velocity field
\begin{equation}\label{uGuT}
u = \nabla^\perp \De^{-1}\om=\nabla^\perp \De^{-1} G + \nabla^\perp \De^{-1}\,\ral \theta \equiv u_G+u_\theta,
\end{equation}
commutator estimates and various functional inequalities.
Proposition \ref{prop:Lq} is crucial in further showing that $G$ is actually
globally regular in the sense that
\begin{equation}\label{Greg}
\|G(t)\|_{B^s_{q,\infty}(\mathbb{R}^2)} \le C \quad \mbox{for}\quad 0\le s\le 3\alpha-2 \quad\mbox{and} \quad 2\le q\le 6,
\end{equation}
and for any $T>0$ and $t\le T$, where $C$ is a constant depending on $T$
and the initial data $(u_0, \theta_0)$. Here $B^s_{q,\infty}$ denotes
an inhomogeneous Besov space (see Section \ref{sec:pre} for more details on Besov spaces). Once Proposition \ref{prop:Lq} is established,
the proof of (\ref{Greg}) is similar to Proposition 7.1 in \cite{JMWZ}. A special
consequence of (\ref{Greg}) is that $G\in B^0_{\infty,1}(\mathbb{R}^2) \subset L^\infty(\mathbb{R}^2)$, which in turn implies that $u_G$ defined in
(\ref{uGuT}) is Lipschitz,
$$
\|\nabla u_G(t)\|_{L^\infty} \le C.
$$
Then (\ref{uGuT}) with the equation of $\theta$ can be treated as a generalized critical SQG equation, which leads to the global regularity of $(u, \theta)$ following a similar approach as in \cite{JMWZ}.

\vskip .1in
The rest of this paper is divided into two main sections. Section \ref{sec:pre} recalls the Littlewood-Paley decomposition,  the definition of Besov spaces and some other related relevant facts.  It also presents several commutator estimates and a global $L^2$-bound for $G$, which serve as a preparation for the proof of Proposition \ref{prop:Lq}. Section \ref{sec:proof} contains the proof of  Proposition \ref{prop:Lq}. In addition,  the proof for Theorem \ref{main} is also given in this section.

\vskip .3in
\section{Preliminaries}
\label{sec:pre}

This section includes several parts. It recalls the Littlewood-Paley theory,
introduces the Besov spaces, provides Bernstein inequalities and Kato-Ponce estimates, proves several commutator estimates and a global $L^2$-bound for $G$.
We start with the definitions of some of the functional spaces and related facts
that will be used in the subsequent sections. Materials on Besov space and related facts
presented here  can be found in several
books and many papers (see, e.g., \cite{BCD,BL,MWZ,RS,Tri}).
\subsection{Fourier transform and the Littlewood-Paley theory}

We start with several notations. $\mathcal{S}$ denotes
the usual Schwarz class and ${\mathcal S}'$ its dual, the space of
tempered distributions. ${\mathcal S}_0$ denotes a subspace of ${\mathcal
S}$ defined by
$$
{\mathcal S}_0 = \left\{ \phi\in {\mathcal S}: \,\, \int_{\mathbb{R}^d}
\phi(x)\, x^\gamma \,dx =0, \,|\gamma| =0,1,2,\cdots \right\}
$$
and ${\mathcal S}_0'$ denotes its dual. ${\mathcal S}_0'$ can be identified
as
$$
{\mathcal S}_0' = {\mathcal S}' / {\mathcal S}_0^\perp = {\mathcal S}' /{\mathcal P},
$$
where ${\mathcal P}$ denotes the space of multinomials. On the Schwartz  class, we can define the Fourier transform and its inverse via
$$
\hat{f}(\xi)=\int_{\mathbb{R}^d}  f(x) e^{-i x \xi} dx, \quad f(x)=\frac{1}{(2\pi)^d} \int_{\mathbb{R}^d}  \hat{f}(\xi) e^{i x \xi} d\xi.
$$

\vskip .1in
To introduce the Littlewood-Paley decomposition, we
write for each $j\in \mathbb{Z}$
\begin{equation*}\label{aj}
A_j =\left\{ \xi \in \mathbb{R}^d: \,\, 2^{j-1} \le |\xi| <
2^{j+1}\right\}.
\end{equation*}
The Littlewood-Paley decomposition asserts the existence of a
sequence of functions $\{\Phi_j\}_{j\in {\mathbb Z}}\in {\mathcal S}$ such
that
$$
\mbox{supp} \widehat{\Phi}_j \subset A_j, \qquad
\widehat{\Phi}_j(\xi) = \widehat{\Phi}_0(2^{-j} \xi)
\quad\mbox{or}\quad \Phi_j (x) =2^{jd} \Phi_0(2^j x),
$$
and
$$
\sum_{j=-\infty}^\infty \widehat{\Phi}_j(\xi) = \left\{
\begin{array}{ll}
1&,\quad \mbox{if}\,\,\xi\in {\mathbb R}^d\setminus \{0\},\\
0&,\quad \mbox{if}\,\,\xi=0.
\end{array}
\right.
$$
Therefore, for a general function $\psi\in {\mathcal S}$, we have
$$
\sum_{j=-\infty}^\infty \widehat{\Phi}_j(\xi)
\widehat{\psi}(\xi)=\widehat{\psi}(\xi) \quad\mbox{for $\xi\in {\mathbb
R}^d\setminus \{0\}$}.
$$
In addition, if $\psi\in {\mathcal S}_0$, then
$$
\sum_{j=-\infty}^\infty \widehat{\Phi}_j(\xi)
\widehat{\psi}(\xi)=\widehat{\psi}(\xi) \quad\mbox{for any $\xi\in
{\mathbb R}^d $}.
$$
That is, for $\psi\in {\mathcal S}_0$,
$$
\sum_{j=-\infty}^\infty \Phi_j \ast \psi = \psi
$$
and hence
$$
\sum_{j=-\infty}^\infty \Phi_j \ast f = f, \qquad f\in {\mathcal S}_0'
$$
in the sense of weak-$\ast$ topology of ${\mathcal S}_0'$. For
notational convenience, we define
\begin{equation}\label{del1}
\mathring{\Delta}_j f = \Phi_j \ast f, \qquad j \in {\mathbb Z}.
\end{equation}
\subsection{Besov spaces}

\begin{define}
For $s\in {\mathbb R}$ and $1\le p,q\le \infty$, the homogeneous Besov
space $\mathring{B}^s_{p,q}$ consists of $f\in {\mathcal S}_0' $
satisfying
$$
\|f\|_{\mathring{B}^s_{p,q}} \equiv \|2^{js} \|\mathring{\Delta}_j
f\|_{L^p}\|_{l^q} <\infty.
$$
\end{define}

\vspace{.1in}
We now choose $\Psi\in {\mathcal S}$ such that
$$
\widehat{\Psi} (\xi) = 1 - \sum_{j=0}^\infty \widehat{\Phi}_j (\xi),
\quad \xi \in {\mathbb R}^d.
$$
Then, for any $\psi\in {\mathcal S}$,
$$
\Psi \ast \psi + \sum_{j=0}^\infty \Phi_j \ast \psi =\psi
$$
and hence
\begin{equation*}\label{sf}
\Psi \ast f + \sum_{j=0}^\infty \Phi_j \ast f =f
\end{equation*}
in ${\mathcal S}'$ for any $f\in {\mathcal S}'$. To define the inhomogeneous Besov space, we set
\begin{equation} \label{del2}
\Delta_j f = \left\{
\begin{array}{ll}
0,&\quad \mbox{if}\,\,j\le -2, \\
\Psi\ast f,&\quad \mbox{if}\,\,j=-1, \\
\Phi_j \ast f, &\quad \mbox{if} \,\,j=0,1,2,\cdots.
\end{array}
\right.
\end{equation}
\begin{define}
The inhomogeneous Besov space $B^s_{p,q}$ with $1\le p,q \le \infty$
and $s\in {\mathbb R}$ consists of functions $f\in {\mathcal S}'$
satisfying
$$
\|f\|_{B^s_{p,q}} \equiv \|2^{js} \|\Delta_j f\|_{L^p} \|_{l^q}
<\infty.
$$
\end{define}

\vskip .1in
The Besov spaces $\mathring{B}^s_{p,q}$ and $B^s_{p,q}$ with  $s\in (0,1)$ and $1\le p,q\le \infty$ can be equivalently defined by the norms
$$
\|f\|_{\mathring{B}^s_{p,q}}  = \left(\int_{\mathbb{R}^d} \frac{(\|f(x+t)-f(x)\|_{L^p})^q}{|t|^{d+sq}} dt\right)^{1/q},
$$
$$
\|f\|_{B^s_{p,q}}  = \|f\|_{L^p} + \left(\int_{\mathbb{R}^d} \frac{(\|f(x+t)-f(x)\|_{L^p})^q}{|t|^{d+sq}} dt\right)^{1/q}.
$$
When $q=\infty$, the expressions are interpreted as supremums instead of integrals.

\vskip .1in
Many frequently used function spaces are special cases of Besov spaces. The following proposition
lists some useful equivalence and embedding relations.
\begin{prop}
For any $s\in \mathbb{R}$,
$$
\mathring{H}^s \sim \mathring{B}^s_{2,2}, \quad H^s \sim B^s_{2,2}.
$$
For any $s\in \mathbb{R}$ and $1<q<\infty$,
$$
\mathring{B}^{s}_{q,\min\{q,2\}} \hookrightarrow \mathring{W}_{q}^s \hookrightarrow \mathring{B}^{s}_{q,\max\{q,2\}}.
$$
In particular, $\mathring{B}^{0}_{q,\min\{q,2\}} \hookrightarrow L^q \hookrightarrow \mathring{B}^{0}_{q,\max\{q,2\}}$.
\end{prop}

\vskip .1in
For notational convenience, we write $\Delta_j$ for
$\mathring{\Delta}_j$. There will be no confusion if we keep in mind that
$\Delta_j$'s associated with the homogeneous Besov spaces is defined in
(\ref{del1}) while those associated with the inhomogeneous Besov
spaces are defined in (\ref{del2}). Besides the Fourier localization operators $\Delta_j$,
the partial sum $S_j$ is also a useful notation. For an integer $j$,
$$
S_j \equiv \sum_{k=-1}^{j-1} \Delta_k,
$$
where $\Delta_k$ is given by (\ref{del2}). For any $f\in \mathcal{S}'$, the Fourier
transform of $S_j f$ is supported on the ball of radius $2^j$.

\subsection{Bernstein inequalities and Kato-Ponce estimates}
Bernstein's inequalities are useful tools in dealing with Fourier localized functions
and these inequalities trade integrability for derivatives. The following proposition
provides Bernstein type inequalities for fractional derivatives.
\begin{prop}\label{bern}
Let $\alpha\ge0$. Let $1\le p\le q\le \infty$.
\begin{enumerate}
\item[1)] If $f$ satisfies
$$
\mbox{supp}\, \widehat{f} \subset \{\xi\in \mathbb{R}^d: \,\, |\xi|
\le K 2^j \},
$$
for some integer $j$ and a constant $K>0$, then
$$
\|\La^{\al} f\|_{L^q(\mathbb{R}^d)} \le C_1\, 2^{\alpha j +
j d(\frac{1}{p}-\frac{1}{q})} \|f\|_{L^p(\mathbb{R}^d)},
$$
where $C_1$ is a constant depending on $K, \alpha,p$ and $q$
only.
\item[2)] If $f$ satisfies
\begin{equation*}\label{spp}
\mbox{supp}\, \widehat{f} \subset \{\xi\in \mathbb{R}^d: \,\, K_12^j
\le |\xi| \le K_2 2^j \}
\end{equation*}
for some integer $j$ and constants $0<K_1\le K_2$, then
$$
C_1\, 2^{\alpha j} \|f\|_{L^q(\mathbb{R}^d)} \le \|\Lambda^\alpha
f\|_{L^q(\mathbb{R}^d)} \le C_2\, 2^{\alpha j +
j d(\frac{1}{p}-\frac{1}{q})} \|f\|_{L^p(\mathbb{R}^d)},
$$
where $C_2$ is a  constant depending on $K_1, K_2, \alpha,p$ and $q$
only.
\end{enumerate}
\end{prop}

\vskip .1in
The following Kato-Ponce type estimate (also known as fractional
Leibnitz rule) will be used extensively,
\begin{equation}
\label{277}
\|\La^s(f g)\|_{L^p}\leq C( \|\La^s f\|_{L^{q_1}} \|g\|_{L^{r_1}}+ \|\La^s g\|_{L^{q_2}} \|f\|_{L^{r_2}}),
\end{equation}
whenever $s>0, 1<p,q_1,q_2<\infty$, $1<r_1, r_2\le \infty$, $\frac{1}{p}=\frac{1}{q_1}+\frac{1}{r_1}=\frac{1}{q_2}+\frac{1}{r_2}$. This inequality can be found in
many references (see  \cite{KPV} and \cite{GO} for recent results and survey of the literature on the topic). As a corollary, we can deduce the following estimate, at least for all integers $m$
\begin{equation}
\label{27}
\|\La^s(f^m)\|_{L^p}\leq C \|\La^s f\|_{L^{q}} \|f\|_{L^{r(m-1)}}^{m-1},
\end{equation}
whenever $1<p,q,r<\infty$ and $\frac{1}{p}=\frac{1}{q}+\frac{1}{r}$.
\subsection{Commutator estimates}
As we have seen already, the equation \eqref{Geqin} involves commutators. Thus, we shall need to develop the corresponding estimates, so that we can bound them suitably.
\begin{lemma}
\label{le:980}
Let $1>\al>1/2$ and $1<p<\infty$, $1<p_1,p_2\leq \infty: \frac{1}{p}=\frac{1}{p_1}+\frac{1}{p_2}$. For any integer $k$, we have for the $u_G$ and $u_\theta$ defined in \eqref{uGuT},
\begin{eqnarray}
\label{80}
\|\Delta_k([\ral, u_G\cdot\nabla]\theta)\|_{L^p}\leq C  2^{(1-\al)k}\|G\|_{L^{p_1}}\|\theta\|_{L^{p_2}};\\
\label{90}
\|\Delta_k([\ral, u_{\theta}\cdot\nabla]\psi)\|_{L^p}\leq C 2^{(2-2\al)k}  \|\theta\|_{L^{p_1}}\|\psi\|_{L^{p_2}}.
\end{eqnarray}
More generally,  for $0\leq s\leq 1-\al$, we have
\begin{eqnarray}
\label{280}
\|\Delta_k([\ral, u_G\cdot\nabla]\theta)\|_{L^p}\leq C_s  2^{(1-\al-s)k}\|G\|_{L^{p_1}}
\|\La^s \theta\|_{L^{p_2}};\\
\label{290}
\|\Delta_k([\ral, u_{\theta}\cdot\nabla]\psi)\|_{L^p}\leq C_s 2^{(2-2\al-s)k}  \|\La^s \theta\|_{L^{p_1}}\|\psi\|_{L^{p_2}}.
\end{eqnarray}

\end{lemma}
\begin{proof}
Most of the proof will be concerned with \eqref{80}, \eqref{90}. At the end, we will indicate the  (small) modifications needed for \eqref{280} and \eqref{290}.

In order to simplify notations, we denote $u_j:=\Delta_j u$, $u_{<k}:=S_k u $, $u_{[k-A, k+B]}:=
\sum_{j=k-A}^{k+B} u_j$. If $A, B<10$, we use $u_{\sim k}$ to denote $u_{[k-A, k+B]}$, etc.
We use the following paraproduct decomposition for the product of two functions
\begin{eqnarray*}
\Delta_k(f g)=   \Delta_k( f_{<k-10}  g_{\sim k})+ \Delta_k(f_{\sim k} g_{<k+10})+\Delta_k(\sum_{l=k+10}^\infty f_l g_{\sim l}).
\end{eqnarray*}
We refer to the first term as low-high interaction, the second term is high-low interaction and the third term is high-high interaction.  Only the low-high interaction term is not straightforward and requires the commutator structure.   Note that according to the definition \eqref{uGuT}, we have that $u_G\sim \La^{-1} G$, $u_{\theta}\sim \La^{-\al} \theta$.\\
 More precisely,  for any $1\leq p\leq \infty$,
\begin{eqnarray*}
\|\Delta_l(u_G)\|_{L^p} &\sim & 2^{-l} \|\Delta_l G\|_{L^p}\leq C 2^{-l} \|G\|_{L^p}; \\
\|\Delta_l(u_\theta)\|_{L^p} & \sim &  2^{-\al l} \|\Delta_l \theta\|_{L^p}\leq C 2^{-\al l} \|\theta\|_{L^p}.
\end{eqnarray*}
 \\
{\bf High-low interactions} \\
For \eqref{80}, we have by H\"older's inequality,
\begin{eqnarray*}
& & \|\Delta_k([\ral, (u_G)_{\sim k}\cdot\nabla ]\theta_{<k+10})\|_{L^p}  \leq C
 \| \ral  \Delta_k [(u_G)_{\sim k}\cdot\nabla \theta_{<k+10}]\|_{L^p}  \\
&+& \|[(u_G)_{\sim k}\cdot\nabla ]\ral \theta_{<k+10}\|_{L^p})\leq C 2^{k(1-\al)} \|(u_G)_{\sim k}\|_{L^{p_1}}
\|\nabla \theta_{<k+10}\|_{L^{p_2}} \\
&+& C   \|(u_G)_{\sim k}\|_{L^{p_1}} \|\La^{2-\al}\theta _{<k+10}\|_{L^{p_2}}.
\end{eqnarray*}
But $\|(u_G)_{\sim k}\|_{L^{p_1}} \leq C 2^{-k} \|G\|_{L^{p_1}}$, and
\begin{equation}
\label{110}
\|\nabla \theta_{<k+10}\|_{L^{p_2}}\leq C 2^k \|\theta\|_{L^{p_2}}, \ \
\|\La^{2-\al} \theta_{<k+10}\|_{L^{p_2}}\leq C 2^{(2-\al)k} \|\theta\|_{L^{p_2}}.
\end{equation}
Putting everything together yields the desired inequality
$$
\|\|\Delta_k([\ral, (u_G)_{\sim k}\cdot\nabla ]\theta_{<k+10})\|_{L^p}  \leq C 2^{k(1-\al)} \|G\|_{L^{p_1}}\|\theta\|_{L^{p_2}}.
$$
Similarly, for \eqref{90}, we have by H\"older's inequality,
\begin{eqnarray*}
& & \| \Delta_k([\ral, (u_{\theta})_{\sim k}\cdot\nabla ]\psi_{<k+10})\|_{L^p}    \leq C
 \| \ral [(u_\theta)_{\sim k}\cdot\nabla \psi_{<k+10}]\|_{L^p}  \\
&+& \|[(u_\theta)_{\sim k}\cdot\nabla ]\ral \psi_{<k+10}\|_{L^p})\leq C 2^{k(1-\al)} \|(u_\theta)_{\sim k} \|_{L^{p_1}}
\|\nabla \psi_{<k+10}\|_{L^{p_2}} \\
&+& C     \|(u_\theta)_{\sim k}\|_{L^{p_1}} \|\La^{2-\al} \psi_{<k+10}\|_{L^{p_2}}.
\end{eqnarray*}
Again, $ \|(u_\theta)_{\sim k} \|_{L^{p_1}} \leq C 2^{-k\al} \|\theta\|_{L^{p_1}}$,  in  conjunction with \eqref{110}, yields
$$
\|\Delta_k([\ral, (u_{\theta})_{\sim k}\cdot\nabla ]\psi_{<k+10})\|_{L^p}    \leq C 2^{2(1-\al)k} \|\theta\|_{L^{p_1}}
\|\psi\|_{L^{p_2}}.
$$
{\bf High-high interactions}\\
For \eqref{80}, we have   that $u\cdot \nabla \theta=\nabla\cdot [u\, \theta]$ (since $\nabla\cdot u=0$). Thus, we have by H\"older's inequality,
\begin{eqnarray*}
& & \|\sum_{l=k+10}^\infty \Delta_k(  [\ral, (u_G)_{l}\cdot\nabla ]\theta_{\sim l})\|_{L^p}  \leq C \sum_{l=k+10}^\infty
 \| \ral \nabla \Delta_k  [(u_G)_{l}\cdot  \theta_{\sim l}]\|_{L^p}  \\
&+& C \sum_{l=k+10}^\infty \|\nabla \Delta_k [(u_G)_{l}\cdot  ]\ral \theta_{\sim l}\|_{L^p})\leq C 2^{k(2-\al)}  \sum_{l=k+10}^\infty \|(u_G)_{l}\|_{L^{p_1}}
\| \theta_{\sim l}\|_{L^{p_2}} \\
&+& C 2^k \sum_{l=k+10}^\infty  \|(u_G)_{l}\|_{L^{p_1}} \|\La^{1-\al}\theta _{\sim l}\|_{L^{p_2}} \\
&\leq &
C\|G\|_{L^{p_1}}\|\theta\|_{L^{p_2}}  \sum_{l=k+10}^\infty (2^{k(2-\al)} 2^{-l}  + 2^k 2^{-\al l})\leq C 2^{k(1-\al)}
\|G\|_{L^{p_1}}\|\theta\|_{L^{p_2}}.
\end{eqnarray*}
For \eqref{90}, we proceed in the same way, but note that  towards the end, we need that $\al>1/2$.  We have
\begin{eqnarray*}
& & \|\sum_{l=k+10}^\infty \Delta_k(  [\ral, (u_\theta)_{l}\cdot\nabla ]\psi_{\sim l})\|_{L^p}  \leq C \sum_{l=k+10}^\infty
 \| \ral \nabla \Delta_k  [(u_\theta)_{l}\cdot  \psi_{\sim l}]\|_{L^p}  \\
&+& C \sum_{l=k+10}^\infty \|\nabla \Delta_k [(u_\theta)_{l}\cdot  ]\ral \psi_{\sim l}\|_{L^p})\leq C 2^{k(2-\al)}  \sum_{l=k+10}^\infty \|(u_\theta)_{l}\|_{L^{p_1}}
\| \psi_{\sim l}\|_{L^{p_2}} \\
&+& C 2^k \sum_{l=k+10}^\infty  \|(u_\theta)_{l}\|_{L^{p_1}} \|\La^{1-\al}\psi _{\sim l}\|_{L^{p_2}}  \\
&\leq &
C\|G\|_{L^{p_1}}\|\theta\|_{L^{p_2}}  \sum_{l=k+10}^\infty (2^{k(2-\al)} 2^{-\al l}  + 2^k 2^{-l(2\al-1)})\leq C 2^{2k(1-\al)}
\|G\|_{L^{p_1}}\|\theta\|_{L^{p_2}}.
\end{eqnarray*}
{\bf Low high interactions}\\
Now we need to estimate
$\|  \Delta_k(  [\ral, (u_G)_{<k-10}\cdot\nabla ]\theta_{\sim k})\|_{L^p}$. As before, we use the divergence free condition to reduce to
\begin{eqnarray*}
\|  \Delta_k(  [\ral, (u_G)_{<k-10}\cdot\nabla ]\theta_{\sim k})\|_{L^p} &=&  \|  \nabla \Delta_k(  [\ral, (u_G)_{<k-10}\cdot]\theta_{\sim k})\|_{L^p}\\
&\leq & C 2^k \| \Delta_k(  [\ral, (u_G)_{<k-10}\cdot]\theta_{\sim k})\|_{L^p},
\end{eqnarray*}
so that now, we need to check
\begin{equation}
\label{120}
\| \Delta_k(  [\ral, (u_G)_{<k-10}\cdot]\theta_{\sim k})\|_{L^p}\leq C 2^{-k\al} \|G\|_{L^{p_1}} \|\theta\|_{L^{p_2}}.
\end{equation}
In addition, $\ral=\partial_1 \La^{-\al}$, so we use the (standard)  product rule to conclude
 \begin{eqnarray*}
 & & [\ral, (u_G)_{<k-10}\cdot]\theta_{\sim k}) = \ral[(u_G)_{<k-10}\cdot \theta_{\sim k}]-(u_G)_{<k-10}\cdot\ral \theta_{\sim k}  \\
&=& \La^{-\al}[\partial_1 (u_G)_{<k-10}\cdot\theta_{\sim k}]+ [\La^{-\al},(u_G)_{<k-10}\cdot] \partial_1\theta_{\sim k}.
\end{eqnarray*}
Clearly, the first term satisfies the required estimates, since
$$
\|\Delta_k \La^{-\al}[\partial_1 (u_G)_{<k-10}\cdot\theta_{\sim k}]\|_{L^p} \leq 2^{-k\al}
\|\partial_1 (u_G)_{<k-10}\|_{L^{p_1}} \|\theta_{\sim k}\|_{L^{p_2}}\leq C 2^{-k\al} \|G\|_{L^{p_1}} \|\theta\|_{L^{p_2}},
$$
which is \eqref{120}. It then remains to show
\begin{equation}
\label{130}
\|\Delta_k [\La^{-\al},f_{<k-10}] g_{\sim k}\|_{L^p}\leq C  2^{-k(1+ \al)} \|\nabla f\|_{L^{p_1}} \|g\|_{L^{p_2}}.
\end{equation}
Indeed, if we show that, we apply it to $f=u_G, g=\partial_1 \theta$ and we obtain the result.

Write
\begin{eqnarray*}
 & & \Delta_k [\La^{-\al},f_{<k-10}] g_{\sim k}=\Delta_k \La^{-\al}[f_{<k-10} g_{\sim k}] -
\Delta_k[f_{<k-10} \La^{-\al} g_{\sim k}]
\end{eqnarray*}
Denote the multiplier of $g_{\sim k}$ by $\tilde{\Delta}_k$. Note that by the support properties of the corresponding multipliers, we have that $\tilde{\Delta}_k \Delta_k=\Delta_k$. Thus,
\begin{eqnarray*}
\Delta_k [\La^{-\al},f_{<k-10}] g_{\sim k} &=& \Delta_k \tilde{\Delta}_k \La^{-\al}[f_{<k-10} g_{\sim k}] -
\Delta_k[f_{<k-10} \La^{-\al} g_{\sim k}]\\
&=& \Delta_k([\tilde{\Delta}_k \La^{-\al}, f_{<k-10}] g_{\sim k}).
\end{eqnarray*}
Thus, it suffices to estimate $\|[\tilde{\Delta}_k \La^{-\al}, f_{<k-10}] g_{\sim k}\|_{L^p}$. Furthermore, we have that
$$
\tilde{\Delta}_k \La^{-\al}=2^{-k\al} P_k,
$$
where $\widehat{P_k f}(\xi)=\tilde{\chi}(2^{-k} \xi) \hat{f}(\xi)$, where $\tilde{\chi}$ is a $C^\infty$ function, supported in $\{\xi: |\xi|\in (1/2, 2)\}$. Thus, we need to show
\begin{equation}
\label{140}
\|[P_k,f] g\|_{L^p}\leq C 2^{-k} \|\nabla f\|_{L^{p_1}} \|g\|_{L^{p_2}}.
\end{equation}
But this is a standard result in harmonic analysis. Here is the easy proof  for completeness
\begin{eqnarray*}
[P_k,f] g(x) &=& 2^{k d} \int_{\rd} \hat{\chi}(2^k(x-y))(f(y)-f(x)) g(y) dy \\
&=& 2^{k d}\int_{\rd} \hat{\chi}(2^k(x-y))  g(y) (\int_0^1 \langle y-x, \nabla f(x-\rho(x-y))\rangle d\rho) dy
\end{eqnarray*}
It follows that
\begin{eqnarray*}
|[P_k,f]g(x)|\leq \int_0^1 \int_{\rd}     | \nabla f(x-\rho z )|   |g(x-z)| 2^{k d} |z| |\hat{\chi}(2^k z))| dz d\rho
\end{eqnarray*}
By H\"older's
$$
\| [P_k,f]g\|_{L^p}\leq C\|\nabla f\|_{L^{p_1}} \|g\|_{L^{p_2}} \int_{\rd} 2^{k d} |z| |\hat{\chi}(2^k z))| dz =
C 2^{-k} \|\nabla f\|_{L^{p_1}} \|g\|_{L^{p_2}}.
$$
This finishes the proof of \eqref{140} and hence of \eqref{80}.

For the low-high interaction term of \eqref{90}, we reduce similarly. More precisely, by the divergence free condition
\begin{eqnarray*}
\|  \Delta_k(  [\ral, (u_\theta)_{<k-10}\cdot\nabla ]\psi_{\sim k})\|_{L^p} &=&  \|  \nabla \Delta_k(  [\ral, (u_\theta)_{<k-10}\cdot]\psi_{\sim k})\|_{L^p} \\
&\leq & C 2^k \| \Delta_k(  [\ral, (u_\theta)_{<k-10}\cdot]\psi_{\sim k})\|_{L^p},
\end{eqnarray*}
Next,
\begin{eqnarray*}
 & & [\ral, (u_\theta)_{<k-10}\cdot]\psi_{\sim k}) = \ral[(u_\theta)_{<k-10}\cdot \psi_{\sim k}]-(u_\theta)_{<k-10}\cdot\ral \psi_{\sim k} \\
&=& \La^{-\al}[\partial_1 (u_\theta)_{<k-10}\cdot\psi_{\sim k}]+ [\La^{-\al},(u_\theta)_{<k-10}\cdot] \partial_1\psi_{\sim k}.
\end{eqnarray*}
For the first term, since $\|\partial_1 (u_\theta)_{<k-10}\|_{L^{p_1}}\leq C 2^{k(1-\al)} \|\theta\|_{L^{p_1}}$,
$$
\|\Delta_k \La^{-\al}[\partial_1 (u_\theta)_{<k-10}\cdot\psi_{\sim k}]\|_{L^p} \leq 2^{-k\al}
\|\partial_1 (u_\theta)_{<k-10}\|_{L^{p_1}} \|\psi_{\sim k}\|_{L^{p_2}}\leq C 2^{k(1-2\al)}
\|\theta\|_{L^{p_1}} \|\psi\|_{L^{p_2}}.
$$
For the second, we can reduce, in a similar way, to proving an  estimate in the form
$$
\|[\tilde{\Delta}_k \La^{-\al}, (u_\theta)_{<k-10}] \psi_{\sim k}\|_{L^p}\leq 2^{-2\al k}
\|\theta\|_{L^{p_1}} \|\psi\|_{L^{p_2}}.
$$
Recalling $\tilde{\Delta}_k \La^{-\al}= 2^{-k\al} P_k $, we have by \eqref{140} and
$\|\nabla (u_\theta)_{<k-10}\|_{L^{p_1}}\leq C 2^{k(1-\al)} \|\theta\|_{L^{p_1}}$,
\begin{eqnarray*}
 & &\|[\tilde{\Delta}_k \La^{-\al}, (u_\theta)_{<k-10}] \psi_{\sim k}\|_{L^p} = 2^{-k\al} \|P_k,  (u_\theta)_{<k-10}] \psi_{\sim k}\|_{L^p} \\
&\leq & C 2^{-k(1+\al)} \|\nabla (u_\theta)_{<k-10}\|_{L^{p_1}} \|\psi\|_{L^{p_2}}\leq C 2^{-2 k\al}
\|\theta\|_{L^{p_1}}  \|\psi\|_{L^{p_2}}.
\end{eqnarray*}
Regarding the proofs of \eqref{280} and \eqref{290}, one just needs to go back to the arguments presented above and trace the derivatives. More precisely, for  \eqref{280}, things are clear in the high-high and the low high interaction cases, since we use \eqref{80} and the inequality
$$
2^{k(1-\al)} \|\theta_{\geq k-10}\|_{L^{p_2}}\leq C 2^{k(1-\al-s)}\|\La^s \theta\|_{L^{p_2}}.
$$
In the high-low interaction case, note that we have used $\|\nabla \theta_{<k+10}\|_{L^{p_2}}\leq 2^k\|\theta\|_{L^{p_2}}$. If we have  instead used
$$
\|\nabla \theta_{<k+10}\|_{L^{p_2}}\leq 2^{k(1-s)}\|\La^s \theta\|_{L^{p_2}},
$$
we would have obtained \eqref{280}, instead of \eqref{80}. The arguments for \eqref{290} are of similar nature and we omit them.
\end{proof}
We now can deduce the following
\begin{corollary}
\label{cor1}
Let $1>\alpha>1/2$ and $1< p_2<\infty, 1<p_1, p_3\leq \infty$, so that \\ $ \f{1}{p_1}+\f{1}{p_2}+\f{1}{p_3}=1$.
For every $s_1:0\leq s_1< 1-\al$ and $s_2: s_2> 1-\alpha-s_1$,  there exists a
$C=C(p_1, p_2, p_3, s_1, s_2)$, so that
\begin{equation}
\label{300}
|\int_{\rd}  F [\ral, u_G\cdot\nabla]\theta dx|\leq C \|\La^{s_1} \theta\|_{L^{p_1}} \|F\|_{W^{s_2,p_2}} \|G\|_{L^{p_3}}
\end{equation}
Similarly, for every $s_1:0\leq s_1< 1-\al$ and $s_2: s_2> 2-2\alpha-s_1$, we have
\begin{equation}
\label{310}
|\int_{\rd}  F [\ral, u_\theta\cdot\nabla]\psi dx|\leq C \|\La^{s_1} \theta\|_{L^{p_1}} \|F\|_{W^{s_2,p_2}}
\|\psi\|_{L^{p_3}}
\end{equation}
\end{corollary}
\begin{proof}
 Recall that we have $\sum_j \mathring{\Delta}_j=Id$. Take $\tilde{\Delta}_j$ with similar properties, so that $\tilde{\Delta}_j \mathring{\Delta}_j=\mathring{\Delta}_j$.  For \eqref{300},  we have by \eqref{280},
\begin{eqnarray*}
& & |\int_{\rd}  F [\ral, u_\theta\cdot\nabla]\psi dx|\leq \sum_j  \int |\tilde{\Delta}_j F|
|\mathring{\Delta}_j[\ral, u_G\cdot\nabla]\theta] |dx\leq \\
&\leq & C \sum_j 2^{j(1-\al-s_1)}
\|\tilde{\Delta}_j F\|_{L^{p_2}} \|G\|_{L^{p_3}} \|\La^{s_1}\theta\|_{L^{p_1}}\leq \\
&\leq &
C  \|G\|_{L^{p_3}} \|\La^{s_1}\theta\|_{L^{p_1}} \max(\|F_{<10}\|_{L^{p_2}}, \sup_{j\geq 0}  2^{js_2}
\|\tilde{\Delta}_j F\|_{L^{p_2}}),
\end{eqnarray*}
which of course implies \eqref{300}.  The proof of \eqref{310} is similar.
\end{proof}

\subsection{$L^2$ bound for $G$}
We  present the global $L^2$-bound for $G$, which improves the
corresponding $L^2$-bound in Theorem 5.1 of \cite{JMWZ} by relaxing the
condition from $\alpha>\frac45$ to $\alpha>\frac34$.

\begin{lemma}
\label{le:106}
Let $ \al>\frac{3}{4}$ and $(u, \theta)$ be the solution of \eqref{BQE} in some interval $[0,T]$. Then,
 $G$ defined in \eqref{Geqin} satisfies for every $0\leq t\leq T$,
$$
\|G(t)\|_{L^2}^2 + \int_0^t \|\La^{\frac{\al}{2}} G(\tau)\|_{L^2}^2 d\tau\leq C(T, u_0, \theta_0).
$$
\end{lemma}

\begin{proof} The proof has some similarities to that for Theorem 5.1
in \cite{JMWZ}, but we make use of the global bound (see \eqref{L2theta})
$$
\int_0^T \|\Lambda^{\frac{\beta}{2}} \theta(\tau)\|_{L^2}^2\,d\tau < C(T, \theta_0).
$$
Taking the inner product of (\ref{Geqin}) with $G$, we obtain, after integration by parts,
\begin{equation}\label{g2root}
\frac12 \frac{d}{dt} \|G\|_{L^2}^2  + \|\Lambda^{\frac{\alpha}{2}} G\|_{L^2}^2  = J_1 + J_2,
\end{equation}
where
\begin{eqnarray*}
J_1 =\int G\,\Lambda^{1-2\alpha} \pp_1 \theta\,dx,\qquad J_2 =\int G\, [\mathcal{R}_\alpha, u\cdot\nabla] \theta \,dx.
\end{eqnarray*}
Applying H\"{o}lder's inequality and noting that the Riesz transform $\Lambda^{-1}\partial_1$ is bounded in $L^q$ for any $1<q<\infty$, we obtain, due to $\frac34<\alpha<1$,
$$
|J_1| \le \|\Lambda^{2-\frac52\alpha} \theta\|_{L^2}\, \|\Lambda^{\frac\alpha2} G\|_{L^2} \le \frac14 \, \|\Lambda^{\frac{\alpha}{2}} G\|_{L^2}^2  \, + C\,\|\theta\|^2_{H^{\frac{\beta}2}}.
$$
To bound $J_2$, we write $u=u_G + u_\theta$ as in (\ref{uGuT}). Again, due to
$\frac34<\alpha<1$, we can choose $1-\al <s<\al/2$ and then apply Corollary \ref{cor1} to obtain
\begin{eqnarray*}
\left|\int G\, [\mathcal{R}_\alpha, u_G\cdot\nabla] \theta \,dx \right|
&\le& C\, \|\theta\|_{L^\infty} \, \|G\|_{L^2}\, \|G\|_{H^s} \le C\, \|\theta_0\|_{L^\infty}\, \|G\|_{L^2} \, \|G\|_{H^{\frac{\al}{2}}}\\
&\le& \frac14 \, \|\Lambda^{\frac{\alpha}{2}} G\|_{L^2}^2  \, + C\, \|G\|_{L^2}^2.
\end{eqnarray*}
Applying Corollary \ref{cor1} with $s_1=\frac{(1-\alpha)}{2}$ and $\frac32(1-\al)<s_2<\frac\al2$, we have
\begin{eqnarray*}
\left|\int G\, [\mathcal{R}_\alpha, u_\theta\cdot\nabla] \theta \,dx \right|
&\le& C\, \|\Lambda^{s_1}\theta\|_{L^2}\,\|\theta\|_{L^\infty}\, \|G\|_{H^{s_2}} \le C\, \|\theta_0\|_{L^\infty}\, \|\Lambda^{\frac{\beta}{2}}\theta\|_{L^2} \, \|G\|_{H^{\frac{\al}{2}}}\\
&\le& \frac14 \, \|\Lambda^{\frac{\alpha}{2}} G\|_{L^2}^2  \, + C\,\|\Lambda^{\frac{\beta}{2}}\theta\|_{L^2}^2.
\end{eqnarray*}
Therefore,
$$
|J_2| \le \frac12 \, \|\Lambda^{\frac{\alpha}{2}} G\|_{L^2}^2  \, + C\, \|G\|_{L^2}^2 + C\,\|\Lambda^{\frac{\beta}{2}}\theta\|_{L^2}^2.
$$
Inserting the bounds for $J_1$ and $J_2$ to (\ref{g2root}) and applying Gronwall's inequality yield the desired bound. This completes the proof of
Lemma \ref{le:106}.
\end{proof}

\vskip .3in
\section{On the $L^6$ bound for $G$}
\label{sec:proof}

This section proves Proposition \ref{prop:Lq}, which provides a
global $L^6$-bound for $G$. Once Proposition \ref{prop:Lq} is established,
Theorem \ref{main} can be proven similarly to \cite{JMWZ}.  We nevertheless  give a brief outline of its proof here.

\vskip .1in
 Before we start proving Proposition \ref{prop:Lq}, let us prepare with the following observation. By the estimates for the evolution of $\theta$ in  \eqref{L2theta}, we have control of $\|\theta(t)\|_{L^\infty}$  and $\|\La^{(1-\al)/2}\theta\|_{L^2_t L^2_x}$. By Gagliardo-Nirenberg's inequality, we have that for all $\gamma\in (0,1/2)$,
\begin{equation}
\label{320}
\|\La^{\ga(1-\al)}\theta\|_{L^{\f{1}{\ga}}_t L^{\f{1}{\ga}}_x}\leq
\|\La^{(1-\al)/2}\theta\|_{L^2_t L^2_x}^{2\ga} \|\theta_0\|_{L^\infty_{t x}}^{1-2\ga}.
\end{equation}

\begin{proof}[Proof of Proposition \ref{prop:Lq}]
We consider $\ga\in (0,1/2)$, to be fixed momentarily and also define $\al_{cr}$ be the solution to
$
(2-\ga)(1-\al)=\f{\al}{2}
$
or
$$
\al_{cr}=\f{4-2 \ga}{5-2\ga}.
$$
Note that for each $\al>\al_{cr}$, we have that $ (2-\ga)(1-\al)<\f{\al}{2}$. We henceforth assume $\al>\al_{cr}$.

In view of Lemma \ref{le:106}, it suffices to consider the case $q=6$.
By multiplying \eqref{Geqin} by $G |G|^{4}=G^{5}$ and integrating in $x$,
we obtain
\begin{equation}
\label{60}
\frac{1}{q} \pp_t\|G(t)\|_{L^6}^6+ \int G^{5}\,\La^{\al} G\,dx =\int G^{5}  [\ral, u\cdot \nabla] \theta dx\, + \int G^{5} \La^{1-2\al} \pp_1 \theta dx.
\end{equation}
By the maximum principle of \cite{CC} and Sobolev embedding, we have
$$
 \int G^{5}\, \La^{\al} G\,dx\geq C \int |\La^{\frac{\al}{2}} G^{3}|^2 dx\geq C\|G\|_{L^{\frac{12}{2-\al}}}^6.
$$
Next, we deal with the  second term on the right hand side of \eqref{60}.  Note that $\pp_1\La^{-1}$ is the Riesz transform in the first variable, which is bounded on all $L^p, 1<p<\infty$ spaces. We have by H\"older's inequality and the Kato-Ponce estimate \eqref{27}
\begin{eqnarray*}
\left|\int G^{5} \La^{1-2\al} \pp_1 \theta dx \right|
&\leq & \|\La^{\ga(1-\al)} \theta\|_{L^{\f{1}{\ga}}_x}
\|\pp_1 \La^{-1} \La^{(2-\ga)(1-\al)}(G^{5})\|_{L^{\frac{1}{1-\ga}}}\\
&\leq & C   \|\La^{\ga(1-\al)} \theta\|_{L^{\f{1}{\ga}}_x} \|\La^{(2-\ga)(1-\al)} G\|_{L^2}
\|G\|_{L^{\f{8}{1-2\ga}}}^4 \\
&\leq & C   \|\La^{\ga(1-\al)} \theta\|_{L^{\f{1}{\ga}}_x} \|G\|_{H^{\al/2}} \|G\|_{L^{\f{8}{1-2\ga}}}^4.
\end{eqnarray*}
where in the last line, we have used that $\al>\al_{cr}$.

\vskip .1in
By writing $u=u_G+u_\theta$ as in (\ref{uGuT}), the first term on
the right hand side of \eqref{60} is split into two terms.
We have, according to  \eqref{310}, for every  $s>(2-\ga)(1-\al)$,
\begin{eqnarray*}
  \left|\int G^{5}  [\ral, u_\theta\cdot \nabla] \theta dx\right| &\leq &  C_s
\|\La^{\ga(1-\al)} \theta\|_{L^{\f{1}{\ga}}}  \|\theta\|_{L^\infty} \|G^5\|_{W^{s,\f{1}{1-\ga}}} \\
&\leq & \|\La^{\ga(1-\al)} \theta\|_{L_x^{\f{1}{\ga}}} \|\theta_0\|_{L^\infty}
\|G\|_{H^s}
\|G\|_{L^{\f{8}{1-2\ga}}}^4.
\end{eqnarray*}
We now pick $s$ so close to $(2-\ga)(1-\al)$, so that $(2-\ga)(1-\al)<s<\al/2$. This is possible, because $\al>\al_{cr}$. Thus,
$$
\left|\int G^{5}   [\ral, u_\theta\cdot \nabla] \theta dx \right|\leq  C   \|\La^{\ga(1-\al)} \theta\|_{L^{\f{1}{\ga}}_x}
\|G\|_{H^{\al/2}} \|G\|_{L^{\f{8}{1-2\ga}}}^4.
$$

\vskip .1in
We now handle the term $\int G^{5}  [\ral, u_G\cdot \nabla] \theta dx$.
By \eqref{300},
$$
\left|\int G^5[\mathcal{R}_\al, u_G\cdot\nabla]\theta\right| \le \|\Lambda^{(1-\gamma)(1-\al)+\rho} (G^5)\|_{L^{p_1}} \, \|G\|_{L^6} \, \|\Lambda^{\gamma\,(1-\al)} \theta\|_{L_x^{\frac1\gamma}},
$$
where $\rho>0$ is a small parameter and
\begin{align}
\frac1{p_1}=\frac56 -\gamma. \label{hh}
\end{align}
By the Kato-Ponce estimate and the Gagliardo-Nirenberg inequality,
\begin{align}
\|\Lambda^{(1-\gamma)(1-\al)+\rho} (G^5)\|_{L^{p_1}}  \le &\,  C\,\|\Lambda^{(1-\gamma)(1-\al)+\rho} G\|_{L^{p_2}} \, \|G\|_{L^{4p_3}}^4 \notag\\
\le & \,C\,\|G\|_{H^{\frac{\al}{2}}}^a \, \|G\|_{L^6}^{1-a}\, \|G\|_{L^6}^{4b} \, \|G\|_{L^{\frac{12}{2-\al}}}^{4(1-b)},
\label{uges}
\end{align}
where $a, b\in (0,1)$ and $p_2, p_3\in (1,\infty)$
satisfy
\begin{align}
& (1-\gamma)(1-\al)+\rho = \frac{\al}{2}\, a, \label{adef}\\
& \frac1{p_2} =
\frac{(1-\gamma)(1-\al)+\rho}{2} +  a \left(\frac12-\frac{\al}{4}\right) +
\frac16(1-a) = \frac{1+2a}{6}, \notag\\
&\frac1{p_3} = \frac1{p_1} -\frac1{p_2} =\frac{2-a}{3} -\gamma, \notag\\
&\frac{1}{4 p_3} = \frac16 b + (1-b) \frac{2-\al}{12}, \quad\mbox{or}\quad b=1-\frac1{\alpha}(a+3\gamma).  \label{bdef}
\end{align}
We return later in the proof to check that $a$ and $b$ satisfy
$a, b\in (0,1)$.
Combining (\ref{320}) and (\ref{uges}) yields
\begin{align*}
&\left|\int G^5[\mathcal{R}_\al, u_G\cdot\nabla]\theta\right| \le C\, \|G\|_{L^{\frac{12}{2-\al}}}^{4(1-b)}\,\|G\|_{H^{\frac{\al}{2}}}^a \,\|\Lambda^{\frac{1-\al}{2}}\theta\|_{L^2_{x}}^{2\gamma}\,\|G\|_{L^6}^{2-a+4b}\\
&\quad \le \frac18\, \|G\|_{L^{\frac{12}{2-\al}}}^6 + C\, \left(\|G\|_{H^{\frac{\al}{2}}}^a \,\|\Lambda^{\frac{1-\al}{2}}\theta\|_{L^2_{x}}^{2\gamma}\right)^{(\frac{3}{2(1-b)})'}\,
\|G\|_{L^6}^{(2-a+4b)(\frac{3}{2(1-b)})'},
\end{align*}
where $(\frac{3}{2(1-b)})'$ denotes the conjugate
index of $\frac{3}{2(1-b)}$.
We collect all the estimates for the right hand
side of \eqref{60} to obtain
\begin{eqnarray*}
\pp_t \|G\|_{L^6}^6+ C \|G\|_{L^{\frac{12}{2-\al}}}^6 &\leq &
C \|\La^{\ga(1-\al)} \theta\|_{L^{\f{1}{\ga}}} \|G\|_{H^{\al/2}} \|G\|_{L^{\f{8}{1-2\ga}}}^4  \\
&+& C\,\left(\|G\|_{H^{\frac{\al}{2}}}^a \,\|\Lambda^{\frac{1-\al}{2}}\theta\|_{L^2_{x}}^{2\gamma}\right)^{(\frac{3}{2(1-b)})'}\,
\|G\|_{L^6}^{(2-a+4b)(\frac{3}{2(1-b)})'}.
\end{eqnarray*}
By the Gagliardo-Nirbenberg  inequality, we have
\begin{eqnarray*}
\|G\|_{L^{\f{8}{1-2\ga}}} &\leq & \|G\|_{L^6}^{\beta_1} \|G\|_{L^{\f{12}{2-\al}}}^{1-\beta_1},
\end{eqnarray*}
where $\beta_1=\beta_1(\ga, \al)$ is determined from
\begin{equation} \label{440}
\f{1-2\ga}{8}=\f{\beta_1}{6}+(1-\beta_1)\f{2-\al}{12} \quad \mbox{or}\quad \beta_1 =  \f{12}{\al}\left(\f{1-2\ga}{8}+\f{\al-2}{12}\right).
\end{equation}
Note that our $\ga, \al$ need to be such that
$\beta_1(\gamma, \al)\in (0,1)$.  We return later in the proof
to check this. We invoke (\ref{320}) and apply Young's inequality to obtain
\begin{eqnarray*}
&&  \|\La^{\ga(1-\al)} \theta\|_{L_x^{\f{1}{\ga}}} \|G\|_{H^{\al/2}} \|G\|_{L^{\f{8}{1-2\ga}}}^4 \\
&& \qquad \le \frac18 \|G\|_{L^{\f{12}{2-\al}}}^{6} +
C\, (\|\Lambda^{\frac{1-\al}{2}} \theta\|^{2\gamma}_{L_{x}^2} \|G\|_{H^{\al/2}} \|G\|_{L^6}^{4\beta_1})^{(\f{3}{2(1-\beta_1)})'}.
\end{eqnarray*}
Thus,
\begin{eqnarray*}
\pp_t \|G\|_{L^6}^6 &\leq & C\,\left(\|G\|_{H^{\frac{\al}{2}}}^a \,\|\Lambda^{\frac{1-\al}{2}}\theta\|_{L^2_{x}}^{2\gamma}\right)^{(\frac{3}{2(1-b)})'}\,
\|G\|_{L^6}^{(2-a+4b)(\frac{3}{2(1-b)})'}  \\
&+& C(\|\Lambda^{\frac{1-\al}{2}}\theta\|^{2\gamma}_{L_{x}^2}
\|G\|_{H^{\al/2}})^{(\f{3}{2(1-\beta_1)})'}
\|G\|_{L^6}^{4\beta_1(\f{3}{2(1-\beta_1)})'}.
\end{eqnarray*}
In order to close the argument, the indices should satisfy
\begin{eqnarray}
\label{400}
& & (2-a+4b)(\frac{3}{2(1-b)})' \le 6; \qquad  4\beta_1(\f{3}{2(1-\beta_1)})' \le 6;\\
\label{410}
& &(a + 2 \gamma) (\frac{3}{2(1-b)})' \le 2; \qquad (\f{3}{2(1-\beta_1)})'\leq \f{2}{2\ga+1}.
\end{eqnarray}
Indeed, by Young's inequality, we have
$$
\left(\|G\|_{H^{\frac{\al}{2}}}^a \,\|\Lambda^{\frac{1-\al}{2}}\theta\|_{L^2_{x}}^{2\gamma}\right)^{(\frac{3}{2(1-b)})'}
\le C(\|G\|_{H^{\al/2}}^2 + \|\Lambda^{\frac{1-\al}{2}}\theta\|_{L^2_{x}}^{2}),
$$
$$
(\|\Lambda^{\frac{1-\al}{2}}\theta\|^{2\gamma}_{L_{x}^2}
\|G\|_{H^{\al/2}})^{(\f{3}{2(1-\beta_1)})'} \le C(\|G\|_{H^{\al/2}}^2 + \|\Lambda^{\frac{1-\al}{2}}\theta\|_{L^2_{x}}^{2})
$$
and
$$
\|G\|_{L^6}^{(2-a+4b)(\frac{3}{2(1-b)})'} \leq C(1 + \|G\|_{L^6}^{6}),
\qquad \|G\|_{L^6}^{4\beta_1(\f{3}{2(1-\beta_1)})'} \leq C(1 + \|G\|_{L^6}^{6}).
$$
This implies the differential inequality
$$
\pp_t \|G\|_{L^6}^6
\leq\,  C(1 + \|G\|_{L^6}^{6})\left(\|G\|_{H^{\al/2}}^2 + \|\Lambda^{\frac{1-\al}{2}}\theta\|_{L^2_{x}}^{2}\right).
$$
Applying Gronwall's then yields
$$
\|G(T)\|_{L^6}^6\leq (1+\|G(0)\|_{L^6}^6) \exp(A(T)),
$$
where
$$
A(T)=\int_0^T \left(\|G\|_{H^{\al/2}}^2 + \|\Lambda^{\frac{1-\al}{2}}\theta\|_{L^2_{x}}^{2}\right) dt <\infty.
$$
This finishes the argument.

\vskip .1in
It remains to analyze the inequalities \eqref{400} and \eqref{410}.
It turns out that \eqref{400} is always satisfied. The first inequality in
\eqref{400} in the same as
$$
2-a + 4b \le 6\left(1-\frac{2(1-b)}{3}\right) =2 + 4b,
$$
which is trivially true for $a>0$. The second inequality in \eqref{400}
is always true as well,
$$
4 \beta_1 \le 6\left(1-\frac23(1-\beta_1)\right) = 4 \beta_1 + 2.
$$
The first condition in \eqref{410} is simplified to
$$
a + 2 \gamma \le \frac23 + \frac43 b,
$$
or, according to (\ref{adef}) and (\ref{bdef}),
$$
\left(1+\frac4{3\alpha}\right) \frac2\al ((1-\gamma)(1-\al)+\rho) + \left(2+\frac4{\alpha}\right) \gamma \le 2.
$$
Since $\rho$ can be taken as small as we wish, we can further reduce this inequality to
$$
6(1-\gamma)\al^2 + (1-7\gamma) \al -4 (1-\gamma) > 0,
$$
which is equivalent to
$$
\al > \al_1(\gamma)\equiv \frac{-(1-7\gamma)+\sqrt{(1-7\gamma)^2+96(1-\gamma)^2}}{12(1-\gamma)}.
$$
The second condition in \eqref{410}, namely, $(\f{3}{2(1-\beta_1)})'\leq \f{2}{2\ga+1}$ is equivalent to $\f{3}{2(1-\beta_1)}\geq \f{2}{1-2\ga}$ or $\beta_1\geq  \f{1+6\ga}{4}$.  Using \eqref{440} for $\beta_1$ yields
$$
\al\geq \f{12\ga+2}{3-6\ga}.
$$

\vskip .1in
This means that $\al$ needs to satisfy the following inequalities
$$
\al > \max\left(\al_1(\gamma), \f{12\ga+2}{3-6\ga}, \f{4-2 \ga}{5-2\ga}\right).
$$
The smallest value of this maximum is achieved  for  $\ga_0=\f{43-\sqrt{1777}}{36}$ so that the value of $\al_{cr}$ is minimized and we get
$$
\al_{cr}=\f{4-2 \ga_0}{5-2\ga_0}=\f{12\ga_0+2}{3-6\ga_0}=\f{\sqrt{1777}-23}{24}=0.798103...
$$
 Finally, recall that we also need  to check that $a, b, \beta_1\in (0,1)$ for $\gamma=\gamma_0$ and $\al\in (\al_{cr},1)$.
The picture below verifies this.
This completes the proof of Proposition \ref{prop:Lq}.
\begin{figure}[ht]
\centering
\includegraphics[width=15cm,height=7cm]{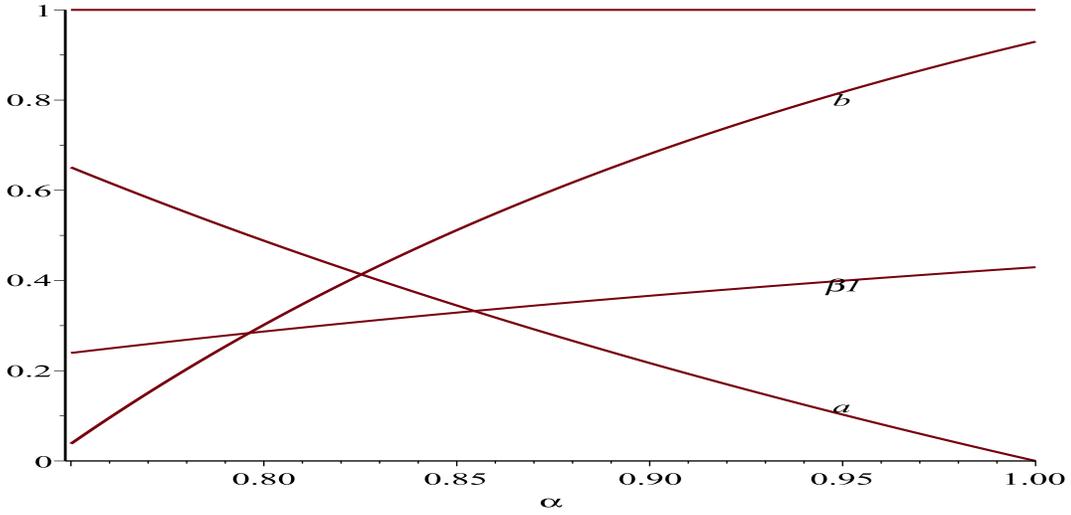}
\caption{Graphics   of $a, b, \beta_1$ for $\gamma=\gamma_0$ and $\al\in (\al_{cr}, 1)$.}
\label{fig}
\end{figure}

\end{proof}

\vskip .1in
We now briefly sketch the proof of Theorem \ref{main}.
\begin{proof}[Proof of Theorem \ref{main}] The global existence and
smoothness of solutions is proven via two steps. The first step is the
local well-posedness of (\ref{BQE}), which can be established through a
standard procedure (see, e.g., \cite{MB,Wu}). The second step extends the
local solution of the first step into a global one through
{\it a priori} estimates. Proposition \ref{prop:Lq} provides a
global $L^q$-bound for $G$ for any $2\le q\le 6$. As in Proposition 7.1
in \cite{JMWZ}, we can show that, for any $0\le s \le 3\alpha-2$,
$$
\sup_{0\leq t\leq T} \|G(t)\|_{B^s_{6, \infty}}\leq C(T, u_0, \theta_0).
$$
Recall the embedding, $B^{3\alpha-2}_{6, \infty}(\mathbb{R}^2)\hookrightarrow
B^0_{\infty,1}(\mathbb{R}^2)$ for $\alpha>\frac79$. Therefore, for $\alpha>\f{\sqrt{1777}-23}{24}>\frac79$
$$
\|\nabla u_G\|_{L^\infty} =\|\nabla \nabla^\perp \De^{-1} G\|_{L^\infty}
\le \|G\|_{B^0_{\infty,1}} \le \|G\|_{B^s_{6, \infty}}.
$$
This yields a global Lipschitz bound on $u_G$. The rest of the proof is the same as in \cite{JMWZ}. We thus omit further details. This completes the proof of
Theorem \ref{main}.
\end{proof}

\vskip .4in
\section*{Acknowledgements}
Stefanov's research is partially supported by NSF grant DMS 1313107. Wu is partially supported by NSF grant DMS 1209153 and the AT\&T Foundation at Oklahoma State University. Wu thanks Professors Quansen Jiu, Changxing Miao and Zhifei Zhang for discussions and thanks Drs. Haifeng Shang and Zhuan Ye for comments.

\end{document}